\documentclass[fleqn, reqno,11pt]{article}
\usepackage[colorlinks]{hyperref}
\AtBeginDocument{
	\hypersetup{
		colorlinks=true,
 		linkcolor=RoyalBlue,
  		citecolor=RedOrange,
	 }
}
\usepackage{amsmath,amsthm,amssymb}
\usepackage[top = 1in, bottom = 1in, left = .75in, right =.75in]{geometry}
\usepackage{enumerate}
\usepackage{verbatim}
\usepackage{graphicx}
\usepackage{tabularx}
\usepackage[usenames,dvipsnames]{xcolor}
\usepackage{mathrsfs}
\usepackage{amsbsy}
\usepackage{pgf,tikz}
\usetikzlibrary{arrows}
\usepackage{bbm}

\numberwithin{equation}{section}

\theoremstyle{plain}
\newtheorem{theorem}{\bf Theorem}[section]
\newtheorem{corollary}[theorem]{\bf Corollary}
\newtheorem{lemma}[theorem]{\bf Lemma}

\newtheorem{question}{\bf Question}

\theoremstyle{definition}

\newtheorem*{example*}{\bf Example}

\newtheorem{definition}[theorem]{Definition}

\theoremstyle{remark}

\newcommand{\mean}{\mathbb{E}}
\newcommand{\pr}{\mathbb{P}}

\newcommand{\ind}[1]{\mathbbm{1}_{\langle #1\rangle}}
\newcommand{\whp}{{\bf whp}}
\newcommand{\Whp}{{\bf Whp}}

\newcommand{\C}{\mathcal{C}}
\newcommand{\D}{\mathcal{D}}

\newcommand{\be}{\begin{equation}}
\newcommand{\ee}{\end{equation}}

\newcommand{\eps}{\varepsilon}
\newcommand{\Len}[1]{\ensuremath{\mathcal{L}_{\ge #1}}}

\newcommand{\lp}{\left(}
\newcommand{\rp}{\right)}
\newcommand{\lb}{\left[}
\newcommand{\rb}{\right]}
\newcommand{\lf}{\lfloor}
\newcommand{\rf}{\rfloor}

\newcommand{\ch}[1]{\langle #1 \rangle}
\newcommand{\one}{c_1}
\newcommand{\Poisson}{\operatorname{Pois}}
\newcommand{\Bin}{\operatorname{Bin}}

\title{On a uniformly random chord diagram and its intersection graph}  
\author{H\"{u}seyin Acan\\ School of Mathematical Sciences\\ Monash University\\ Melbourne, VIC 3800\\ Australia\\ 
\texttt{huseyin.acan@monash.edu}
}
\date{}

\begin{document}
\maketitle

\begin{abstract}
A chord diagram refers to a set of chords with distinct endpoints on a circle. The intersection graph of a chord diagram $\cal C$ is defined by substituting the chords of $\cal C$ with vertices and by adding edges between two vertices whenever the corresponding two chords cross each other. Let $C_n$ and $G_n$ denote the chord diagram chosen uniformly at random from all chord diagrams with $n$ chords and the corresponding intersection graph, respectively.
We analyze $C_n$ and $G_n$ as $n$ tends to infinity. In particular,  we study the degree of a random vertex in $G_n$, the $k$-core of $G_n$, and the number of strong components of the directed graph obtained  from $G_n$ by orienting edges by flipping a fair coin for each edge. We also give two equivalent evolutions of a random chord diagram and show that, with probability approaching $1$,  a chord diagram produced after $m$ steps of these evolutions becomes monolithic as $m$ tends to infinity and stays monolithic afterward forever.


\end{abstract}



\section{Introduction}\label{sec:intro}

A \textit{chord diagram} of size $n$ is a pairing of $2n$ given points on a circle. We label the points $1$ through $2n$ clockwise and join each pair of points in the pairing by a chord to obtain a geometric and a combinatorial object. It is easy to see that there are $(2n)!/(2^nn!)$ chord diagrams with $n$ chords.

Although we are mainly concerned about the combinatorics of random chord diagrams in this paper, chord diagrams appear extensively in some other fields such as in the study of some invariants in knot theory~\cite{APRW13,Bar-Nat95,BR00,Kon93,Sto98}, in the representation theory of Lie algebras~\cite{Cam-Stur05}, and in codifying the pairings among nucleotides in RNA molecules~\cite{Bon08RNA,OZ02RNA, Rei10book}. For detailed information about chord diagrams and their topological and algebraic significance we refer the reader to Chmutov, Duzhin, and Mostovoy's book~\cite{CDMbook}. 

Many probabilistic and enumerational problems about chord diagrams have been studied. 
A remarkable formula for the generating function counting chord diagrams with a given genus was given by Harer and Zagier~\cite{HZ86}. Later, Linial and Nowik~\cite{LN11} and Chmutov and Pittel~\cite{CP13} studied the genus of a uniformly chosen random chord diagram. In random graph theory, Bollob\'{a}s et al.~\cite{BRST01} 
used linearized chord diagrams to generate a preferential attachment random graph introduced by Barab\'{a}si and Albert~\cite{BA99}. 

The enumeration of chord diagrams was first studied by Touchard in a sequence of papers. In~\cite{Tou52}, he found a functional equation for the bivariate generating function $\sum_{n,m}T_{n,m}x^my^n$, where $T_{n,m}$ denotes the number of chord diagrams with $n$ chords and $m$ crossings. Using this, J. Riordan~\cite{JRiordan75} found an exact formula for $T_{n,m}$ in the form of an alternating sum.
Note that $T_{n,0}$ counts the famous Catalan numbers.
Using some recurrence relations, Stein and Everett~\cite{SE78} showed that a random chord diagram of size $n$ is connected with probability approaching 1 as  $n\to \infty$.
Cori and Marcus~\cite{CM98} counted non-isomorphic chord diagrams. Acan and Pittel~\cite{AP14+} studied a phase transition for the appearance of a giant component and found asymptotic estimates for $T_{n,m}$ for $m =O(n\log n)$. Various enumerational problems were also studied in~\cite{AcanPhD}.

Flajolet and Noy~\cite{FN00} studied a uniformly random chord diagram with $n$ chords and, using generating functions, they showed that (i) the number of components (defined formally below) approaches $1+\Poisson(1)$ as $n\to \infty$, where $\Poisson(\lambda)$ denotes the Poisson distribution with parameter $\lambda$, and (ii) the number of crossings is asymptotically Gaussian. They also showed that almost all chord diagrams are monolithic, where a monolithic diagram consists of a root component and a number of isolated chords.

In this paper, we study several characteristics of a random chord diagram. In particular we extend the results of Flajolet and Noy about the components of a random chord diagram in several directions. Before we proceed, we note that the results in this paper are presented in graph theory language. The following definition allows us to do so.

\begin{definition}[Intersection graphs]\label{def:graph}
For each chord diagram $\C$, we define the corresponding \textit{intersection graph} $G_{\C}$ as follows. Each chord in $\C$ becomes a vertex in $G_{\C}$, and two vertices in $G_{\C}$ are adjacent if and only if the corresponding chords cross each other in $\C$; see Figure~\ref{fig: CD and graph}. 
\end{definition}

The intersection graphs of chord diagrams are related to circle graphs, where a circle graph is the intersection graph of a set of labeled chords. (The endpoints of chords are not labeled in this case.) Unlabeled versions of these two classes of graphs are the same. Some of the NP-complete problems in general graphs, such as finding  the clique number or independence number, have polynomial time algorithms for circle graphs (and hence for the intersection graphs of chord diagrams).

By Definition~\ref{def:graph}, any graph theoretic term about a chord diagram can be understood in reference to the corresponding graph. 
For example, components of a chord diagram $\C$ corresponds to the components of $G_{\C}$ and the $k$-core of $G_{\C}$ corresponds to the $k$-core of $G_{\C}$.
Throughout the paper, we denote by $C_n$ a chord  diagram chosen uniformly at random from all chord diagrams of size $n$ and (with an abuse of notation) by $G_n$ the corresponding intersection graph.

In section~\ref{se: degrees} we find the asymptotic distribution of the degree of a random chord in $C_n$, that is, the distribution of the number of chords crossing a random chord.
In Section~\ref{section: MCD} we define monolithic chord diagrams and we give an alternative proof for the fact that almost all graphs are monolithic, a fact first proven by Flajolet and Noy. In Section~\ref{sec: k-core} we study the $k$-core of $C_n$ (equivalently the $k$-core of $G_n$) for $k=o\big(\sqrt n\big)$. In Section~\ref{sec: directed chord diagrams} we study oriented chord diagrams or equivalently directed intersection graphs. In particular, we show that the number of strong components in the random directed intersection graph converges in distribution to $1+\Poisson(3)$, a result analogous to the result of Flajolet and Noy for the undirected case.
In Section~\ref{sec: growth of C_n} we present two models of dynamically growing random chord diagrams. We show that these two models are the same in some sense; after $n$ chords are drawn, both of them give $C_n$ after some canonical relabeling of the endpoints. We show that, during these evolutions, as $n\to \infty$ however slowly, a random chord diagram becomes monolithic and then stays monolithic afterward with probability approaching 1.

We conclude this work with a discussion on the independence number of $C_n$. 
Note that a chord diagram can be viewed as a fixed-point-free involution of $[2n]$. 
Using this interpretation and a result of Baik and Rains~\cite[Theorem 3.1]{Ba-Ra01} about the longest increasing subsequence of a random involution, Chen et al.~\cite[Remark 5.6]{Chen07} determined the asymptotic distribution of the clique number of $C_n$. 
Their terminology differs from ours though.
 They define $r$-crossings and $r$-nestings as follows: a set of  $r$ chords $\ch{x_1,y_1},\dots,\ch{x_r,y_r}$ is an $r$-crossing if $x_1<\cdots<x_r<y_1<\cdots<y_r$ and it is an $r$-nesting if $x_1<\cdots<x_r<y_r<\cdots<y_1$. According to this definition, Chen et al.\ showed that crossing numbers and nesting numbers are distributed symmetrically. Recently, Baik and Jenkins~\cite{Ba-Jen13} proved that $cr_n$ and $ne_n$ are  asymptotically independent, where $cr_n$ denotes the maximum crossing in $C_n$ and $ne_n$ denotes the maximum nesting in $C_n$.

\begin{figure}[t]
\centering
\vspace {-3cm}
\begin{tabular} {c c}
\begin{tikzpicture}[scale=1] 

  \coordinate (center) at (0,0);
  \def\radius{1.5cm}
  \draw (center) circle[radius=\radius];
\begin{scriptsize}
\foreach \x in {1,...,10}
{
\path (center) ++(180+0.1*180-\x*0.2*180:\radius) coordinate (A\x);
\fill (A\x) circle[radius=2pt] ++(180+0.1*180-\x*0.2*180:.7em) node {\x};
}
\foreach \from/\to in {A1/A4, A2/A7, A3/A6, A5/A9, A8/A10}
  \draw (\from) -- (\to);
\end{scriptsize}
\end{tikzpicture}

& \qquad 

\begin{tikzpicture}
 [scale=.7,auto=left,every node/.style={ inner sep=2pt}]
\coordinate (n1) at (0,0) {(3,5)};
\coordinate (n2) at (1,2) {(1,4)};
\coordinate (n3)  at (2,0) {(2,7)};
\coordinate (n4)   at (3,2) {(6,8)};
\coordinate (n5) at (4,0) {(2,9)};

\fill (n1) circle[radius=3pt] node[below] {$\ch{1,4}$};
\fill (n2) circle[radius=3pt] node[above] {$\ch{2,7}$};
\fill (n3) circle[radius=3pt] node[below] {$\ch{3,6}$};
\fill (n4) circle[radius=3pt] node[above] {$\ch{5,9}$};
\fill (n5) circle[radius=3pt] node[below] {$\ch{8,10}$};

\foreach \from/\to in {n1/n2, n1/n3, n3/n4, n4/n5, n2/n4}
  \draw (\from) -- (\to);
\end{tikzpicture}
\end{tabular}
\caption{A chord diagram and the corresponding intersection graph}
\label{fig: CD and graph}
\end{figure}
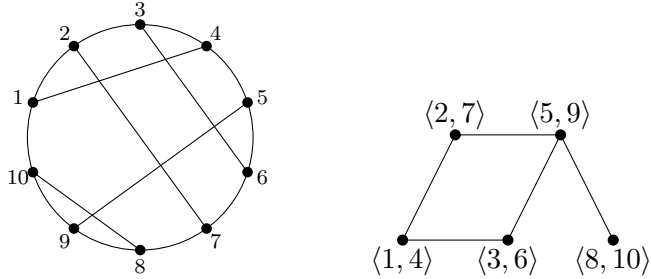

\subsection*{Notation and Terminology}

Here we give the notation and terminology used frequently in the paper. Additional notation and definitions will be given as they become necessary in the later section.
For two integers $a$ and $b$, we denote by $[a,b]$ the set $\{a,a+1,\dots,b\}$. We simply write $[b]$ for $[1,b]$. For a positive integer $k$, we denote by $(2k-1)!!$ the product of odd integers from 1 to $2k-1$.

A chord diagram is a pairing of a set of points, which are referred to as {\em endpoints} in this work. Our endpoints are labeled with positive integers. Unless otherwise stated, a chord diagram of size $n$ has the set of endpoints~$[2n]$. A \emph{block} refers to a set of consecutive endpoints on the circle. The {\em length of a block} is the number of endpoints it contains. A \textit{subdiagram} of a chord diagram $\C$ is a subset $S$ of the chords of $\C$, where the labeling of endpoints is inherited from $\C$. For example, $\{\ch{1,4},\ch{3,6}\}$ is a subdiagram of the chord diagram given in Figure \ref{fig: CD and graph}.

A chord joining the endpoints $x$ and $y$ is denoted by $\langle x,y\rangle$ or $\ch{y,x}$. 
A chord $\ch{x,y}$ divides the circle into two arcs and hence defines two  blocks of endpoints, $B_1$ and $B_2$, where $B_1$ and $B_2$ are disjoint and their union gives the whole set of endpoints except $x$ and $y$. In Figure \ref{fig: CD and graph}, the two blocks determined by the chord $\ch{3,6}$ are $\{4,5\}$ and $\{7,8,9,10,1,2\}$. 
If the set of endpoints is $[2n]$ and $x<y$, then the lengths of the blocks determined by $\ch{x,y}$ is $y-x-1$ and $2n-y+x-1$. 

For a chord $c=\ch{x,y}$ in a given chord diagram $\C$, the length of $c$, denoted $\ell(c)$, is the minimum of the lengths of the two blocks determined by $c$. For example, in Figure~\ref{fig: CD and graph}, the length of $\ch{5,9}$ is $3$.
According to this definition, a chord joining two consecutive endpoints has length $0$. The length of a chord can be at most $n-1$, in which case the chord is $\ch{x,x+n}$ for some $x \in [n]$.


\section{Degree of a random chord in $C_n$}\label{se: degrees}

Let $\C$ be a chord diagram and $c=\ch{x,y}$ be a chord in $\C$. The \textit{degree} of $c$ in $\C$, denoted $deg_{\C}(c)$, is the number of chords crossing $c$ in $\C$.
In other words, the degree of $c$ in $\C$ is the degree of the corresponding vertex in the graph $G_{\C}$.

Studying the degree of a random chord in $C_n$ is equivalent to studying the degree of the chord containing the endpoint 1. In this section this chord is denoted by $\one$.
The length of $\one$ in $C_n$ is distributed as
\[
\pr(\ell(\one)=k) =\begin{cases}
\frac{2}{2n-1} &  \text{ if } 0\le k\le n-2	,		\\
\frac{1}{2n-1} & \text{ if } k=n-1 .
\end{cases}
\]
The degree of $\one$ is closely related to $\ell(\one)$ as we shall see. One obvious observation is that $deg(\one)\le \ell(\one)$. Conditioned on the event $\{\ell(\one)=k\}$, we define $X_k$ as the number of chords with both endpoints lying in the smaller block determined by $\one$. If $\one=\ch{1,n+1}$, then both blocks have lengths $n-1$, in which case it does not matter which block we choose. Note that  $0\le X_k \le \lf k/2\rf$.

\begin{example*}
For $\C = \{\ch{1,8},\ch{2,4},\ch{3,11}, \ch{5,7}, \ch{6,9}, \ch{10,12}\}$, we have $\ell(\one)=4$ and $X_4=1$.
\end{example*}

Now, we compute the mean $\mu_k:=\mean[X_k]$ and the variance $\sigma_k^2:=\mean[X_k^2]-~\mu_k^2$. To find these, we can assume without loss of generality that $\one =\langle 1,k+2\rangle$. 
For the mean, we write
\[	 \mu_k = \sum_{(x,y)} \mean[\ind{x,y}],		\]
where the sum ranges over all $x$ and $y$ such that $2\le x<y\le k$, and $\ind{x,y}$ is the indicator random variable for the existence of the chord $\ch{x,y}$ conditioned on $\one$. Thus, 
\[	 \mean[\ind{x,y}]=\pr\big(\ch{x,y} \text{ is a chord in } C_n \mid \one= \langle 1, k+2\rangle\big)=\frac{(2n-5)!!}{(2n-3)!!}=\frac{1}{2n-3},		\]
and consequently
\be \label{eq: meanX_k}
\mu_k = {k \choose 2}\frac{1}{2n-3}= \frac{k(k-1)}{2(2n-3)} \sim \frac{k^2}{4n}
\ee
as $k$ and $n$ tend to infinity. Thus, as $n$ tends to infinity and for $k=o(\sqrt{n})$, with high probability\footnote{We say that an event $E$ that is defined for all $n\in \{1,2,\dots\}$ occurs  ``\emph{with high probability}", abbreviated as \whp, if the probability of $E$ approaches 1 as $n$ tends to $\infty$. Thus,  even if it is not explicitly stated, we always have the assumption ``$n\to \infty$'' when we use the term ``\whp''.} (\whp),
the degree of $\one$ is $k$, conditioned on $\ell(\one)=k$. To compute the variance, first we write
\[	\mean[X_k(X_k-1)]= \sum_{((x,y),(u,v))} \mean[\ind{x,y}\ind{u,v}] = {k \choose 2}{k-2 \choose 2} \frac{(2n-7)!!}{(2n-3)!!}.		\]
Using the last two equations we get
\[ 	  \mean[X_k^2]= \mean[X_k(X_k-1)]+\mean[X_k]= \frac{k(k-1)(k^2-5k+4n-4)}{4(2n-3)(2n-5)}	\]
and 
\begin{align}\label{eq: varianceX_k}
\sigma_k^2 &= \frac{k(k-1)}{2(2n-3)}\cdot \lp \frac{k^2-5k+4n-4}{2(2n-5)}- \frac{k(k-1)}{2(2n-3)} \rp \notag	\\
&= \frac{k(k-1)}{2(2n-3)^2(2n-5)}\lb (2n-k)^2-10n+5k+6 \rb  	
\le \ \frac{k^2(2n-k)^2}{n^3}. 
\end{align}

\begin{theorem}
Let $D(n)= deg(\one)/n$ and let $b$ be a constant such that $0\le b \le 1/2$. We have,
\[	 \lim_{n \to \infty}\pr(D(n)\le b)= 1- \sqrt{1-2b}.			\]
In other words, $D(n)$ converges in distribution to a random variable with density function $(1-2x)^{-1/2}$ in the interval $[0,0.5]$.
\end{theorem}

\begin{proof}
We want to show that $\lim_{n\to \infty}\pr(deg(\one)\le bn) = 1- \sqrt{1-2b}$. Note that, conditioned on $\ell(\one)=k$, the degree of $\one$ is equal to $k-2X_k$. The mean and the variance of $X_k$ together imply that $X_k$ is concentrated around its expected value $\mu_k$. More specifically, by Chebyshev's inequality and equations \eqref{eq: meanX_k} and \eqref{eq: varianceX_k},
\begin{equation} \label{Cheb}
\pr\lp |X_k-\mu_k| \ge \frac{\omega k}{\sqrt{n}}\rp \le \frac{\sigma_k^2 n}{\omega^2 k^2}\le \frac{(2n-k)^2}{n^2\omega^2} \le \frac{4}{\omega^2} \to 0
\end{equation}
for any $\omega$ approaching infinity. In the rest of the proof, we let $\omega=\log n$ although any $\omega$ approaching $\infty$ sufficiently slowly would work. Now fix a constant $b\in (0,1/2)$. We have
\begin{align*}
\pr(deg(\one) \le bn) &= \sum_{k= 0}^{n-1} \pr(k-2X_k \le bn\  |\  \ell(\one)=k)\cdot  \pr(\ell(\one)=k)	\\
&= \frac{2}{2n-1} \sum_{k=0}^{n-2} \pr(k-2X_k \le bn \ |\  \ell(\one)=k) \\
&\ \ + \frac{1}{2n-1} \cdot \pr(n-1-2X_{n-1} \le bn \  | \  \ell(\one)=n-1).
\end{align*}
Clearly, the last term in this equation is of order $O(1/n)$. Noting that the random variable $X_k$ is defined conditional on the event $\{\ell(\one)=k\}$, we write $\pr \lp (k-bn)/2 \le X_k \rp$ instead of $\pr \lp k-2X_k \le bn\  |\  \ell(\one)=k \rp$ for simplicity. Thus,
\begin{align*}
\pr(deg(\one) \le bn) &=\frac{2}{2n-1}\cdot \sum_{k=0}^{n-2} \pr \lp \frac{k-bn}{2} \le X_k \rp  + O(1/n)\\
&=  \frac{1}{n}\cdot \sum_{k=0}^{n-2} \pr \lp \frac{k-bn}{2} \le X_k\rp +O(1/n).
\end{align*}
Next, we split this sum into two as
\begin{align*}
 \sum_{k=0}^{n-2} \pr \lp \frac{k-bn}{2} \le X_k \rp 
=&  \sum_{k=0}^{n-2}  \pr \lp \frac{k-bn}{2} \le X_k \bigg| |X_k-\mu_k|\le \frac{k\log n}{\sqrt{n}}\rp \cdot \pr\lp |X_k-\mu_k|\le \frac{k\log n}{\sqrt{n}}\rp \\
&\quad +\  \sum_{k=0}^{n-2} \pr \lp \frac{k-bn}{2} \le X_k \bigg| |X_k-\mu_k|> \frac{k\log n}{\sqrt{n}}\rp \cdot \pr\lp |X_k-\mu_k|> \frac{ k\log n}{\sqrt{n}} \rp.
\end{align*}
Since $\pr(\,|X_k-\mu_k|> k\log n/\sqrt n\,) = O((\log n)^{-2})$ by~\eqref{Cheb}, the last sum above is $O(n(\log n)^{-2})$. On the other hand, the first sum on the right side of the equation is equal to
\be \label{eq: firstsum}
 \lp 1-O(1/(\log n)^2)\rp \sum_{k=0}^{n-2} \pr \lp \frac{k-bn}{2} \le X_k \, \bigg| \, |X_k-\mu_k|\le \frac{ k\log n}{\sqrt{n}}\rp.
\ee
Note that
\[	\pr \lp \frac{k-bn}{2} \le X_k \, \Big|  \, |X_k-\mu_k|\le \frac{ k\log n}{\sqrt{n}}\rp =0			\]
unless $k$ belongs to $A$, where
\be	\label{eq: set A_k }
	A:=	 \left\{ k\, : \frac{k-bn}{2}-\frac{k\log n}{\sqrt{n}} \le \mu_k \right\}.
\ee
Thus, the sum in Equation~\eqref{eq: firstsum} is bounded above by the size of $A$, i.e.,
\be 
\sum_{k=0}^{n-2} \pr \lp \frac{k-bn}{2} \le X_k \bigg| |X_k-\mu_k|\le \frac{k\log n}{\sqrt{n}}\rp 	 \le  |A_k|.
\ee
Now let $t:=k/n$, where $t\in [0,1]$. Using~\eqref{eq: meanX_k}, the inequality $\frac{k-bn}{2} - \frac{k\log n}{\sqrt{n}} \le \mu_k$ can be written as
\[
	\frac{t-b}{2}- \frac{t\log n}{\sqrt{n}} - \frac{t(tn-1)}{2(2n-3)} \le 0.
\]
Multiplying by 4 and rearranging, we get
\be  \label{eq: t,1} 
-t^2+2t-2b -\frac{4t\log n}{\sqrt{n}}+\frac{4t-6t^2}{4n-6} \le 0.  		
\ee
The roots of the equation $-t^2+2t-2b=0$
are $1-\sqrt{1-2b}$ and $1+\sqrt{1+2b}$. Moreover, the left side of this equation is negative for $t\in [0,1-\sqrt{1-2b}\,)$ and positive for $(1-\sqrt{1-2b}\,,1]$. Thus, by~\eqref{eq: t,1} 
\[
|A| \le (1-\sqrt{1-2b}+\eps_n)\, n
\]
for any $\eps_n \to 0^+$ as long as $\log n/(\eps_n \sqrt n) \to 0$.
Consequently,
\begin{align} \label{eq: degreeupper}
\pr(deg(\one) \le bn) &\le  \lp 1-O\lp 1/(\log n)^2 \rp \rp \cdot \lp 1-\sqrt{1-2b}\rp+ O\lp 1/(\log n)^2 \rp \notag 		\\
&= 1-\sqrt{1-2b}\,+\, O\lp 1/(\log n)^2 \rp.
\end{align}
To find a lower bound on
\[	\sum_{k=0}^{n-2} \pr \lp \frac{k-bn}{2} \le X_k \bigg| |X_k-\mu_k|\le \frac{k\log n}{\sqrt{n}}\rp,			\]
we note that 
\[	 \pr \lp \frac{k-bn}{2} \le X_k \bigg| |X_k-\mu_k|\le \frac{k\log n}{\sqrt{n}}\rp=1			\]
if $k$ lies in the set $B$, where
\be	\label{eq: set B_k }
	B:=	 \left\{ k\, : \frac{k-bn}{2}+\frac{k\log n}{\sqrt{n}} \le \mu_k \right\}.
\ee
Thus,
\be 
\sum_{k=0}^{n-2} \pr \lp \frac{k-bn}{2} \le X_k \bigg| |X_k-\mu_k|\le \frac{k\log n}{\sqrt{n}}\rp 	 \ge  |B|.
\ee
Now we need to estimate the size of $B$. As in the case of $A$, we let $t=k/n$ and write the inequality in the definition of $B$ in terms of $t$ as
\be  \label{eq: t,2} -t^2+2t-2b +\frac{4t\log n}{\sqrt{n}}+\frac{4t-6t^2}{4n-6} \le 0.  		 \ee
For $0\le t \le1-\sqrt{1-2b}-\eps_n$, the inequality in~\eqref{eq: t,2} is satisfied as long as $\eps_n$ goes to $0$ sufficiently slowly. In particular,~\eqref{eq: t,2} holds for $\eps_n=1/(\log n)^2$ and large enough $n$.
Consequently, analogous to~\eqref{eq: degreeupper}, we have
\be \label{eq: degreelower}
\pr(deg(c) \le bn) \ge  1-\sqrt{1-2b}\,-\, O\lp 1/(\log n)^2 \rp.
\ee
Finally, combining~\eqref{eq: degreeupper} and \eqref{eq: degreelower}, we get
\[  
\lim_{n\to\infty} \pr(deg(c) \le bn)	 =	1-\sqrt{1-2b},	
\] 
which finishes the proof.
\end{proof}


\section{Monolithic chord diagrams}\label{section: MCD}

Monolithic diagrams were introduced by Flajolet and Noy~\cite{FN00}. Using generating functions, they showed that almost all intersection graphs of chord diagrams have one large component and some isolated vertices. Here we give an alternative proof to this result and in the subsequent sections we give several extensions. 

\begin{definition}\label{def: monolithic diagrams}
A chord that connects two consecutive endpoints is called a \textit{simple chord}. Thus, a simple chord is of the form $\ch{i,i+1}$ for some $i\in [2n]$ (addition is always in modulo $2n$). The component containing endpoint 1 is called the \emph{root component}. A chord diagram $\C$ is \textit{monolithic} if 
\vspace{-\topsep}
\begin{enumerate}[(i)] \itemsep0pt \parskip0pt \parsep0pt 
\item $\C$ consists only of the root component and simple chords, and
\item there is no pair of simple chords next to each other in $\C$, that is, there is no $i\in [2n]$ such that both $\ch{i,i+1}$ and $\ch{i+2,i+3}$ are chords in $\C$.
\end{enumerate}
\end{definition}

The original definition given by Flajolet and Noy does not have part (ii) but we include that for convenience in the next sections. As we will see it is not much of a  restriction. 

\begin{theorem}[Flajolet and Noy]\label{FN}
\Whp, $C_n$ is monolithic. Moreover, the number of simple chords in $C_n$ approaches in distribution a Poisson random variable with parameter 1. 
\end{theorem}

\begin{proof}
If a chord diagram $\C$ is not monolithic, then for some $2 \le k \le n/2$, there is a set of $k$ chords whose endpoints form a block of length $2k$. Let $B_k$ denote the number of such sets of $k$ chords in $C_n$. We want to show that $\sum B_k \to 0$, where the sum is over all integer $k$ such that $2\le k \le n/2$. By Markov's inequality, it is enough to show  $\sum \mean[B_k] \to 0$. Since there are $2n$ blocks of $2k$ endpoints, we have 
\[	 \mean[B_k]=2n\frac{(2k-1)!!(2n-2k-1)!!}{(2n-1)!!}	,	\]
and consequently,
\[	\sum_{k=2}^{\lf n/2\rf}\mean[B_k] = 2n \sum_{k=2}^{\lf n/2\rf}\frac{(2k-1)!!(2n-2k-1)!!}{(2n-1)!!}.				\]
Note that, since
\[	
\frac{\mean[B_k]}{\mean[B_{k+1}]}= \frac{2n-2k-1}{2k+1}\ge 1 ,			
\]
$\mean[B_k]$ is decreasing with $k$ for  $k\le \lf n/2\rf -1$. Hence, 
\begin{align*}	
\sum_{k=2}^{\lf n/2\rf}\mean[B_k] = \mean[B_2]+(n/2)\mean[B_3]   \ 	
\le \frac{6n}{(2n-1)(2n-3)} + (n/2) (2n)\frac{5!!(2n-7)!!}{(2n-1)!!} =O\left(\frac 1n\right),
\end{align*}
which shows that $C_n$ is monolithic \whp. 
Finally, the next lemma shows that the number of simple chords converges in distribution to $\Poisson(1)$ as $n\to \infty$.
\end{proof}

Let $L_j$ the number of length $j$ chords in $C_n$. A chord $c$ has length $j$ if $c=\ch{i,i+j+1}$ for some $i\in [2n]$, where the addition is in modulo $2n$. Note that $L_0$ counts the simple chords in $C_n$.
  
\begin{lemma}\label{Y_j}
As $n$ tends to infinity, the random variable $L_j$ converges in distribution to a Poisson random variable with mean $1$, for any $0\le j\le n-2$. 
\end{lemma}

\begin{proof}
We compute the factorial moments $E_r$, where $E_r:=\mean \lb {L_j \choose r} \rb$. Let $\xi_i$ denote the indicator of the event $\{\ch{i,i+j+1} \in C_n\}$.
We have
\[
E_r= \sum_{1\le i_1<i_2\cdots<i_r\le 2n} \pr(\xi_{i_t}=1 \text{ for all } t\in [r] ).		
\]
For  a tuple $(i_1,\dots,i_r)$ in the sum, we have 
\be \label{probability r tuple}
\pr \lp  \prod_{j=1}^r\xi_{i_j}=1 \rp = 
\begin{cases}
0, & \text{if } i_b= i_a+j+1 \text{ for some } a,b\in [r],  \\
\frac{(2n-2r-1)!!}{(2n-1)!!} , & \text{otherwise.}
\end{cases}  
\ee
Hence,
\be \label{upperE_r}	
E_r \le {2n \choose r} \frac{(2n-2r-1)!!}{(2n-1)!!} \to \frac{1}{r!}\,.
\ee
If the probability in~\eqref{probability r tuple} is 0, we call the tuple a \emph{bad tuple}. We obtain an upper bound on the number of bad tuples by choosing an endpoint $i$, then choosing the endpoint $i+j+1$, and then choosing $r-2$ endpoints from the remaining $2n-2$ endpoints. This can be done in $2n{2n-2 \choose r-2}$ ways. Thus,
\be \label{lowerE_r}
E_r\ge \lb {2n\choose r} - 2n{2n-2 \choose r-2}\rb \frac{(2n-2r-1)!!}{(2n-1)!!} \to \frac{1}{r!}.
\ee
By~\eqref{upperE_r}--\eqref{lowerE_r}, the binomial moments of $L_j$ converge to those of a  Poisson random variable with mean 1 and hence $L_j$ converges in distribution to a Poisson random variable with mean 1.
\end{proof}

\section{The $k$-core of $C_n$}\label{sec: k-core}

The $k$-core of a chord diagram $\C$ is the largest subdiagram $S$ of $\C$ with the property that each chord in $S$ crosses at least $k$ other chords in $S$. This means that the minimum degree in the intersection graph of $S$ is at least $k$. In this section we study the size of the $k$-core of $C_n$.

For a chord diagram $\C$, we denote by $\Len{k}(\C)$ the subdiagram of $\C$ consisting of the chords with lengths of at least $k$. For the simplicity of notation, we write \Len{k}\  instead of $\Len{k}(C_n)$. We say that two chords of a chord diagram are \emph{neighbors} if they cross each other.

\begin{theorem}\label{thm: k-core main}
Let $k$ be a function of $n$ such that $k^2=o(n)$ as $n\to \infty$. Then, as $n\to \infty$, the minimum degree in the subdiagram $\Len{k}$ is at least $k$ \whp. Consequently, as $n\to \infty$, the $k$-core of $C_n$ is the subdiagram  $\Len{k}$ \whp.
\end{theorem}

\begin{proof}
Since for any chord $c$ we have $deg(c)\le \ell(c)$, the $k$-core of $C_n$ is a subset of $\Len{k}$. Hence, for $k$ as in the theorem, we only need to show that each chord in $\Len{k}$ has at least $k$ neighbors from $\Len{k}$.
 
Let us call a chord $\ch{x,y}$ of $\Len{k}$ a \emph{bad chord} if $\ch{x,y}$ has fewer than $k$ neighbors in \Len{k}.
It is  enough to show that the expected value for the number of bad chords tends to 0.
For any $t\ge k$, let $r=r(t)$ be the first integer larger than $(t-k)/2$, i.e.,
\be \label{eq: r def}
r=r(t):= \lf (t-k+2)/2 \rf	.		
\ee
Thus, $r(t+1)=r(t)$ if $(t-k)$ is even and $r(t+1)=r(t)+1$ otherwise.
 
Any chord with one endpoint in $[i+1,i+t]$ and one endpoint outside $[i-k,i+t+1+k]$ has length at least $k$. Hence, 
if the chord $c=\ch{i,i+t+1}$ is bad, then, other than $c$, there are at least $r$ chords with both endpoints in
$[i-k,i+t+1+k]$. So the number of bad chords of length $t$ is bounded above by the  number of pairs $(c,\{c_1,\dots,c_r\})$, where $c=\ch{i,i+t+1} \in C_n$ for some $i$,  $\{c_1,\dots,c_r\}$ is a set of chords such that
all the endpoints of $c_j$'s lie in $[i-k,i+t+1+k]$, and $c \not \in \{c_1,\dots,c_r\}$. For $t\ge k$, let $E_t$ be the sum of expected values of such pairs as $i$ varies from $1$ to $2n$. Since there are $2k+t$ endpoints in $[i-k,i+t+1+k]$ other than $i$ and $i+t+1$, 
we have
\be \label{eq: expected bad pairs}	E_t = (2n){2k+t \choose 2r} \frac{(2r-1)!!(2n-2r-3)!!}{(2n-1)!!}.			\ee

We will show that $\sum_{t=k}^{n-1}E_t \to 0$ as $n\to \infty$. First, by~\eqref{eq: r def} we have  $t= 2r+k-2$ or $t=2r+k-1$, depending on the parity of $t-k$. In particular, we have $t\le 2r+k$ for any $t$. Using this upper bound for $t$ in~\eqref{eq: expected bad pairs}, we write
\[	E_t\le 	(2n) {3k+2r \choose 2r}\, \frac{(2r-1)!!(2n-2r-3)!!}{(2n-1)!!}.	\]
Since $r$ is increasing with $t$, the maximum value of $r$ occurs at $t=n-1$ and $r(n-1)=\lf (n+1-k)/2\rf$. Let $N=\lf (n+1-k)/2\rf$. Since each $r$ (except possibly for $r=N$) has two preimages, we write
\[
\sum_{t=k}^{n-1}E_t \le 2\,(2n)\  \sum_{r=1}^{N}  {3k+2r \choose 2r}\frac{(2r-1)!!(2n-2r-3)!!}{(2n-1)!!}	= 4n\  \sum_{r=1}^{N} F_r\, ,
\]
where 
\[	F_r:=	{3k+2r \choose 2r}\frac{(2r-1)!!(2n-2r-3)!!}{(2n-1)!!}.		\]
An easy calculation shows that $F_1 \le 2k^2/n^2$. We will show that this is the main contribution to the sum.
In order to see how the terms $F_r$ are changing, we take the ratio of two consecutive terms and obtain
\begin{equation}\label{gammar}	
\gamma_r:=\frac{F_r}{F_{r+1}} = \frac{(2r+2)(2n-2r-3)}{(3k+2r+2)(3k+2r+1)}.		
\end{equation}
From this fraction we see that the terms $F_r$ are initially decreasing rapidly, and as $r$ gets closer to the upper bound of the sum they almost stabilize. To be more precise, we divide the set $[1,N ]$ into three disjoint pieces $I_1, I_2$, and $I_3$ such that
\[	I_1= \big [1,N_1 \big], \quad I_2= \big[N_1+1, N_2 \big], \quad I_3= \big[N_2+1, N \big]	,	\]
where $N_1=\lf n/10\rf $ and $N_2=\lf 4n/9\rf $. Using~\eqref{gammar}, it is easy to verify that
\[
\gamma_r \ge
\begin{cases}
8, & \text{ for } r\in I_1,\\
1, & \text{ for } r\in I_2,\\
1/2, & \text{ for } r\in I_3.
\end{cases}
\]
In the first interval the terms are decreasing faster than a geometric sequence with ratio $1/8$, so we have $F_{N_1}\le (1/8)^{N_1-1}\, F_1 \text{ and }$
\[	 \sum_{r \in I_1}F_r \le 2F_1. 	\]
In the second interval the terms are decreasing, so we have $F_{N_2}\le F_{N_1} \text{ and }$
\[	 \sum_{r \in I_2}F_r \le n\cdot F_{N_1} \le n(1/8)^{N_1-1}\cdot F_1.	\]
In the last interval we just bound the sum above by a geometric series with a constant ratio 2. As a result, we have
\[	\sum_{r\in I_3}F_r= O(F_{N_2}\cdot 2^{N-N_2})= O(F_1\cdot (1/8)^{N_1}\cdot 2^{N-N_2})		\]
Combining all three sums, we get 
\[	\sum_{r=1}^n F_r =O(F_1)= O(k^2/n^2).		\]
Consequently,
\[	\sum_{t=k}^{n-1}E_t = O(k^2/n),	\]
which goes to $0$ by the assumption of the theorem.
\end{proof}

Recall from Section~\ref{section: MCD} that $L_j$ denotes the number of chords of length $j$.
The next result is an extension of Lemma~\ref{Y_j}.

\begin{theorem}\label{MDP-SC}
Let $k$ be any nonnegative fixed integer. We have 
\[		(L_0,L_1,\dots,L_k) \stackrel{d}{\longrightarrow} (T_0,T_1,\dots,T_k), 	\]
where $T_i$'s are independent copies of the Poisson random variable with parameter $1$.
\end{theorem}

\begin{proof}
We will show that the factorial moments converge to those of independent $\Poisson(1)$ random variables.
Let  $\mathbf r=(r_0,\dots,r_k)$ be a vector of nonnegative integers, and let $s=\sum_{i=0}^{k}r_i$. Let 
\[	 E_{\mathbf r}:= \mean \lb {L_0 \choose r_0}{L_1 \choose r_1}\cdots {L_k \choose r_k}\rb.	\]
To choose a chord of length $j$, it is enough to choose the initial endpoint; if $i$ is the initial endpoint, then the chord of length $j$ will be $\ch{i,i+j+1}$, where we take the value of $i+j+1$ in modulo $2n$ in case $i+j+1>2n$. Hence,
\be \label{upper E_r}
E_{\mathbf{r}} \le {2n \choose s}{s \choose r_0,\dots,r_k} \frac{(2n-2s-1)!!}{(2n-1)!!}. 
\ee
On the right side of this inequality, the first factor is an over-count for the number of ways to choose the initial endpoints of the chords. Once we choose the initial endpoints of the chords, we can partition them into $k+1$ sets, where the endpoints in the $j$th set will correspond to the initial endpoints of length-$j$ chords. Next, there are $(2n-2s-1)!!$ ways to pair the remaining endpoints. The denominator $(2n-1)!!$ is the number of all chord diagrams. After cancellations, we can write the right side of Equation~\eqref{upper E_r} as 
\be \label{upper again}
\lp \prod_{i=0}^{k} \frac{1}{r_i!}\rp \cdot \lp \prod_{j=0}^{s-1}\frac{2n-j}{2n-2j-1}\rp= \lp 1+O\lp \frac{s^2}{n}\rp \rp \cdot \prod_{i=0}^{k} \frac{1}{r_i!}.
\ee
For a lower bound on $E_{\bold r}$, consider the points satisfying the inequalities
\be \label{first points rule}
	1\le i_1<i_2-(k+1)<\cdots<i_s-(s-1)(k+1)\le 2n-s(k+1).
\ee
Note that any $s$-tuple $(i_1,\dots,i_s)$ satisfying~\eqref{first points rule} can be the initial endpoints of the $s$ chords, regardless of how they are partitioned into $k+1$ sets of cardinalities $r_0,\dots,r_k$. Introducing $j_t=i_t-(t-1)(k+1)$, we have
\[	1\le j_1<j_2<\cdots<j_s\le 2n-s(k+1), 	\]
and hence the number of ways to choose the initial points meeting the condition in~\eqref{first points rule} is ${2n-s(k+1) \choose s}$. Thus,
\[	
 E_{\bold r} \ge {2n-s(k+1) \choose s} {s \choose r_0,\dots,r_k} \frac{(2n-2s-1)!!}{(2n-1)!!}. 
\]
Similar to~\eqref{upper again}, for the right side of the above inequality, we have
\be \label{lower again}
\lp \prod_{i=0}^{k} \frac{1}{r_i!}\rp \cdot \lp \prod_{j=0}^{s-1}\frac{2n-s(k+1)-j}{2n-2j-1}\rp= \lp 1+O\lp \frac{s^2k}{n}\rp \rp \cdot \prod_{i=0}^{k} \frac{1}{r_i!}.
\ee
Combining the upper and lower bounds for $E_{\bold r}$, we get
\[	 E_{\bold r} \to \prod_{i=0}^{k} \frac{1}{r_i!},	\]
from which the desired result follows.
\end{proof}

The next theorem is an extension of Theorem~\ref{FN}, which gives the $1$-core of $C_n$.

\begin{corollary}\label{cor: FN first extension}
For any fixed positive integer $k$, the size of the $k$-core of $C_n$ converges in distribution to $n- P_k$, where $P_k$ is a Poisson random variable with mean~$k$.
\end{corollary}

\begin{proof}
By Theorem \ref{thm: k-core main}, the subdiagram $\Len{k}$ is the $k$-core of $C_n$ \whp. By Theorem~\ref{MDP-SC}, the number of chords with lengths smaller than $k$, i.e., the sum $\sum_{j=0}^{k-1}L_j$, converges in distribution to the sum of $k$ independent $\Poisson(1)$ random variables. The sum of independent $\Poisson(\lambda_j)$ random variables is another Poisson random variable with mean $\sum_{j}\lambda_j$, which finishes the proof.
\end{proof}

Next we find an upper bound for the number of `small-length chords'. For a positive integer $k$, let $Z_k=\sum_{j=0}^{k-1}L_j$.

\begin{lemma}\label{small length chords} 
If $k\to \infty$ and $k\le n-1$, then $Z_k/k \to 1$ in probability.
\end{lemma}

\begin{proof}
We find the first and the second moments of $Z_k$ and apply Chebyshev's inequality.
For $j<n-1$, there are exactly $2n$ chords of length $j$, which are $\ch{i,i+j+1}$ for $1\le i\le 2n$. Each of these chords exists in $C_n$ with probability  $1/(2n-1)$. Then, $\mean[L_j]= 2n/(2n-1)$, and consequently
\be \label{meanZk}	
\mu:= \mean[Z_k]= \sum_{j=0}^{k-1}\mean[L_j] = \frac{2nk}{2n-1} = k+\frac{k}{2n-1}.		
\ee
Now, $\mean[Z_k^2]=\mean[Z_k(Z_k-1)]+\mean[Z_k]$ and 
\[ 
\mean[Z_k(Z_k-1)]= \sum_{\ch{a,b},\ch{c,d}}\mean\big[\ind{a,b}\ind{c,d}\big], 
\]
where the sum is over all pairs of distinct chords of lengths at most $k-1$. A chord of length at most $k-1$ can be chosen in $2nk$ ways. Once such a chord, say $\ch{a,b}$, is chosen, there are $2nk-(4k-1)$ ways to choose another chord $\ch{c,d}$ with length at most $k-1$ since $\{a,b\} \cap \{c,d\} = \emptyset$. After these two chords are chosen, the rest of the endpoints can be paired in $(2n-5)!!$, so
\be \label{meanZk2}
\mean[Z_k(Z_k-1)]= \frac{2nk(2nk- 4k+1)(2n-5)!!}{(2n-1)!!}= \frac{2nk(2nk-4k+1)}{(2n-1)(2n-3)}.  
\ee
Using~\eqref{meanZk} and \eqref{meanZk2}, after straightforward computations we get
\[   
\sigma^2:= \mean[Z_k^2]-\mu^2=k+O\lp \frac{k^2}{n}\rp.
\]
Let $\eps$ be a positive constant. Using $|\mu-k|\le 1$ and Chebyshev's inequality,
\[	
\pr\big ( |Z_k-k|>\eps k\big) \   \le    \ \pr\big(|Z_k-\mu|> \eps k-1 \big)  \   \le  \ \frac{\sigma^2}{(\eps k-1)^2}\ =\  O\lp \frac{1}{k}\rp, 
\]
which finishes the proof.
\end{proof}

Lemma \ref{small length chords} combined with Theorem~\ref{thm: k-core main} gives us the approximate size of the $k$-core in $C_n$ for $k$ growing to infinity slowly enough.

\begin{corollary}
Let $k \to \infty$ slowly enough so that $k^2/n\to 0$ as $n$ tends to infinity. Let $R_k$ denote the size of the $k$-core in $C_n$. Then, $(n-R_k)/k \to 1$ in probability as $n\to \infty$. 			\qed
\end{corollary}


\section{Directed intersection graphs}\label{sec: directed chord diagrams}

In this section we consider chord diagrams whose crossings are oriented. For any chord diagram $\C$, we think of the chords as pieces of strings. If two chords $c$ and $d$ of $\C$ cross each other, then either $c$ over-crosses $d$, or $d$ over-crosses $c$. If $c$ over-crosses $d$, then we also say that $d$ under-crosses $c$, and vice versa. With this interpretation, each crossing can be oriented in one of two ways. Given a chord diagram $\C$ with $m$ crossings, there are $2^m$ orientations of $\C$. 

The directed intersection graph $G_{\D}$ is obtained from an oriented chord diagram $\D$ as follows: if a chord $c\in \D$ over-crosses a chord $d \in \D$, then the edge joining $c$ and $d$ in $G_{\D}$ is directed toward the vertex $d$. We say that an oriented chord diagram $\D$ is {\em strongly connected} if $G_{\D}$ is strongly connected as a graph.

We denote by $D_n$ an oriented chord diagram obtained from $C_n$ by choosing one of the possible orientations uniformly at random. More precisely, we choose a chord diagram uniformly at random from all chord diagrams, and then we orient each of its crossings in one of the two possible ways by flipping a fair coin. If $\D$ is an oriented chord diagram, then 
\[
\pr(D_n=\D) = \frac{1}{(2n-1)!!2^{cr(\D)}}\,,
\]
where $cr(\D)$ denotes the number of crossings in $\D$.

The Flajolet-Noy Theorem (Theorem \ref{FN}) tells us that almost all (unoriented) chord diagrams are monolithic. Consequently, as $n\to \infty$, \whp, the intersection graph of $C_n$ has a giant component of size $n-X_n$, where $X_n$ is the number of isolated vertices. The same theorem also gives that $X_n$ converges in distribution to $\Poisson(1)$. Here we will prove an analogous theorem for the random directed intersection graph. We first introduce some terminology that will be helpful later.

Let $\D$ be an oriented chord diagram, $c$ be a chord belonging to $\D$, and $S$ be a subset of the chords of $\D$, where $S$ does not contain $c$. We say that $c$ over-crosses (under-crosses) $S$ if $c$ over-crosses (under-crosses) any chord $d$ of $S$ that it crosses. When $c$ over-crosses $S$ in $\D$, in the intersection graph $G_{\D}$ any edge $cd$ with $d \in S$ is directed toward $d$. Similarly, when  $c$ under-crosses $S$ in $\D$, in the intersection graph $G_{\D}$ any edge $cd$ with $d \in S$ is directed toward $c$. 

In this section we denote by $\Len{k}$ the set of chords with length at least $k$ in $D_n$. In several steps we will show that, \whp, $\Len{k}$ is strongly connected as a subdiagram of $D_n$ as $k$ tends to infinity. We start with finding a large clique in $C_n$, which will act like a progenitor of the `giant' strong component in $G_{D_n}$.

\begin{definition}
A \textit{balanced clique} of size $m$ is a clique such that there is exactly one endpoint in each of the blocks $[1,r], [r+1,2r], \cdots, [(2m-1)r+1, 2mr]$, where $r= \lf n/m \rf$.
\end{definition}

Let $B_m$ denote the number of balanced cliques of size $m$ in $C_n$. We next show that $B_m$ is positive \whp\ for some large $m$.

\begin{lemma}\label{balanced clique}
Let $\eps$ be a positive constant. If $m\le \lf (2+\eps)^{-1}(\log n)^{-1/2}n^{1/2}\rf$, then there exists a balanced clique of size $m$ in $C_n$ \whp, i.e., $ \pr(B_m \ge 1)\to 1 $ as $n\to \infty$.
\end{lemma}

\begin{proof}
Let $r= \lf n/m \rf$, and let $I_t$ denote the block $[(t-1)r+1,tr]$. We have $n/m\ge (2+\eps)\sqrt{n\log n}$, so $r\ge (2+\eps)\sqrt{n\log n}-1$. For $1\le t\le m$, let $A(t)$ be the event that there is no chord $\ch{x,y}$ in $C_n$ such that $x\in I_t$ and $y\in I_{t+m}$. The event $\{B_m=0\}$ lies in the union $\cup_{t=1}^{m}A(t)$. Let $p_t:=\pr(A(t))$. By the union bound,
\[          
\pr(B_m=0) \le \sum_{t=1}^m p_t\,,		
\]
and by symmetry all the $p_t$'s are the same. Thus, $\pr(B_m=0)\le mp_t$. It is enough to show $p_t=o(1/m)$. For $0\le j\le r/2$, we denote by $B(t,j)$  the event that there is no chord from $I_t$ to $I_{t+m}$, and there are exactly $j$ chords with both endpoints in $I_t$. Note that $B(t,j)$'s are disjoint and $A(t)=\cup_{j=0}^{\lf r/2\rf}B(t,j)$. Thus,
\begin{align*}	
p_t &= \sum_{j=0}^{\lf r/2\rf}\pr(B(t,j))	\\
&=\sum_{j=0}^{\lf r/2\rf} {r \choose 2j}(2j-1)!!{2n-2r \choose r-2j}(r-2j)!\frac{(2n-2r+2j-1)!!}{(2n-1)!!}.	
\end{align*}
Let $f_j$ be the term in the above sum corresponding to index $j$. For $j<\lf r/2\rf$ we have
\be \label{ratio of f's}	\frac{f_j}{f_{j+1}}	= \frac{(2j+2)(2n-3r+2j+2)}{(r-2j)(r-2j-1)}\cdot \lp 1-\frac{r}{2n-2r+2j+1}\rp.		\ee
It is easy to see that this ratio is increasing with $j$. Thus, the sequence $\{f_j\}_{j=0}^{\lf r/2\rf}$ is unimodal; there is an index $k$ such that $\{f_j\}$ is increasing for $j\le k$ and it is decreasing afterward. To find this $k$, we need to solve the equation $f_j/f_{j+1}=1$. It is easy to see that the index $k$ of the maximum term goes to infinity and $k=o(r)$.
For $1 \ll j\ll r$ (here $a\ll b$ means $a/b\to 0$), using~\eqref{ratio of f's} we write
\[ 	
\frac{f_j}{f_{j+1}}	\sim \frac{4nj}{r^2}\sim \frac{4j}{(2+\eps)^2 \log n}.	
\]
From this we conclude that the maximum term $f_k$ occurs when $k$ is asymptotic to $(\log n)/2$. Let us find an upper bound for $f_k$. We have
\begin{align}\label{eq:fk}
f_k 	&=	{r \choose 2k}(2k-1)!!{2n-2r \choose r-2k}(r-2k)!\frac{(2n-2r+2k-1)!!}{(2n-1)!!}	\notag 	\\
	&=	\frac{r!}{(r-2k!)}\cdot \frac{1}{2^kk!}\cdot \frac{(2n-2r)!!}{(2n-3r+2k)!!}\cdot \frac{(2n-2r+2k-1)!!}{(2n-1)!!} \notag	\\
	&=	\lp \frac{1}{2^kk!}\rp \cdot \lp \prod_{i=0}^{2k-1}(r-i) \rp \cdot \lp \prod_{i=0}^{r-2k-1}(2n-2r-i) \rp \cdot \lp \prod_{i=0}^{r-k-1} \frac{1}{2n-1-2i}\rp.
\end{align}
Now we find an upper bound for each of the four terms in the product above. Using Stirling's formula, we obtain
\[	 \frac{1}{2^kk!} \le \lp \frac{e}{2k}\rp ^k.		\]
For the second term in~\eqref{eq:fk}, we have
\[	\prod_{i=0}^{2k-1}(r-i)\le r^{2k}.		\]
For the third term in~\eqref{eq:fk}, we have
\begin{align*}
\prod_{i=0}^{r-2k-1}(2n-2r-i)	&=	(2n)^{r-2k} \prod_{i=0}^{r-2k-1}\lp 1-\frac{2r+i}{2n}\rp 	\\
							&=	(2n)^{r-2k} \exp\lb \sum_{i=0}^{r-2k-1}\log \lp 1-\frac{2r+i}{2n}\rp\rb 		\\
							&=	(2n)^{r-2k} \exp \lp -\frac{5r^2}{4n}+O\lp \frac{rk}{n}+\frac{r^3}{n^2}\rp\rp.
\end{align*}
Finally, for the last term in~\eqref{eq:fk}, we have
\begin{align*}
\prod_{i=0}^{r-k-1} \frac{1}{2n-1-2i}		&=		(2n)^{-r+k} \prod_{i=0}^{r-k-1}\lp 1+\frac{2i}{2n-1-2i}\rp 		\\
									&\le 	(2n)^{-r+k} \prod_{i=0}^{r-k-1}\lp 1+\frac{3i}{2n}\rp 			\\
									&= 	 	(2n)^{-r+k} \exp\lp \frac{3r^2}{4n} + O\lp \frac{rk}{n}+\frac{r^3}{n^2}\rp\rp.
\end{align*}
Combining \eqref{eq:fk} and the four upper bounds for the terms on its right side, we get
\[ f_k\le \lp \frac{er^2}{4nk}\rp ^k e^{-r^2/2n}\exp\lp O\lp \frac{rk}{n}+\frac{r^3}{n^2}\rp\rp\le  2\lp \frac{er^2}{4nk}\rp ^k e^{-r^2/2n}.\]
Now letting $g(x)=x \log \lp \frac{er^2}{4nx}\rp$ and taking the derivative, we get
\[ 	 g'(x)= \frac{er^2}{4nx}-1.		\]
Thus, $g(x)$ takes its maximum when $x=r^2/(4n)$, and consequently, the maximum of $\lp \frac{er^2}{4nx}\rp ^x$ occurs for $x=r^2/(4n)$. So,
\[	 \lp \frac{er^2}{4nk}\rp ^k e^{-r^2/2n} \le e^{r^2/4n} e^{-r^2/2n}= e^{-r^2/4n}.	\]
Since $p_t= \sum_{j=0}^{\lf r/2\rf }f_j \le \lb 1+(r/2)\rb f_k$, we have 
\[ p_t= O\lp \sqrt{n \log n}e^{-r^2/(4n)}\rp	\]
and consequently,
\[		 \pr(B_m=0)		\le	 m\cdot p_t = O\lp ne^{-r^2/(4n)}\rp=  O\lp n^{1-\frac{(2+\eps)^2}{4}}\rp, 	\]
which finishes the proof.
\end{proof}

Let $T_n$ denote the tournament chosen uniformly at random from the set of all tournaments on $n$ vertices. In other words, $T_n$ is the directed graph obtained from the complete (undirected) graph $K_n$ by orienting the edges independently of each other with probability $1/2$ in each direction. It is well known that $T_n$ is strongly connected \whp. In fact, a result of
 Janson~\cite[Theorem 1]{Jan95} states that there are many directed Hamilton cycles in $T_n$ \whp.
In the next lemma we use this fact and Lemma~\ref{balanced clique} to show that $\Len{m}$ is strongly connected for sufficiently large $m$ \whp.
Subsequently, in several steps, we will show that the same event holds \whp\ for any $m\to \infty$.
Let $A(k,n)$ denote the  event that there is a strongly connected balanced clique of size $k$ in $D_n$.

\begin{lemma}[Step 1]\label{k_1}
\Whp, the subdiagram $\Len{\lf n^{3/5}\rf}$ is strongly connected as $n\to \infty$.
\end{lemma}

\begin{proof}
Let $m=\lf n^{3/5}\rf$ and $m'= \lf n^{9/20}\rf$. By Lemma~\ref{balanced clique} and the fact that a random tournament is strongly connected \whp, the event $A(m',n)$ holds \whp.
Let $c$ be a chord in $D_n$ whose length is at least $m$. 
Since $\ell(c)\ge m$, each of the two arcs determined by $c$ contains at least $(m-2r)/r$ of the blocks $I_1,\dots,I_{2m'}$, where $r=\lf n/m' \rf$ and $I_t=[(t-1)r+1,tr]$. 
As a result, $c$ intersects at least, say, $m/r-2$ chords of any balanced clique of size $m'$. 
If $B$ is a stongly connected balanced clique of size $m'$ in $D_n$, then $c$ belongs to the same strong component containing $B$ unless either $c$ over-crosses $B$ or it under-crosses $B$. Thus, the probability of $c$ not belonging to the strong component of $B$ is at most $2(1/2)^{m/r-2}$. 
Consequently, the probability that there is a chord of length at least $m$ that does not belong to the strong component containing $B$ is $O\big( n\cdot 2^{-m/r}\big)$, which goes to zero since $m/r \sim n^{1/20}$.
 In other words, conditioned on the event $A(m',n)$, which occurs \whp, all the chords of length at least $m$ lie in the same strong component \whp. This finishes the proof.
\end{proof}

Lemma~\ref{k_1} tells us that the subdiagram consisting of chords of length at least $n^{3/5}$ is strongly connected \whp.
By Lemma \ref{small length chords}, \whp, the number of chords with length at most $n^{\frac{1}{2}+\eps}$ is smaller than $(1+\eps)n^{\frac{1}{2}+\eps}$ for any fixed $\eps>0$. These two results together imply that, \whp,  there is a strong component, which deserves to be called the \emph{giant component}, of size $n-o(n)$. Next, using the previous result, we will show that chords of length at least $n^{1/3}$ belong to this giant component. We start with the following lemma.

\begin{lemma}\label{intervals-chords}
As $n\to \infty$, \whp, there is no block of length $3\lf n^{3/5}\rf $ that contains at least $\lf n^{1/3}\rf$ chords. 
\end{lemma}

\begin{proof}
Let $t:=3\lf n^{3/5}\rf$ and $s:=\lf n^{1/3}\rf$. Let $I_j$ denote the set $\{j,j+1,\dots,j+t-1\}$ of size $t$, and $p(I_j)$ denote the probability that $I_j$ contains at least $s$ chords. By symmetry, $p(I_i)=p(I_j)$ for any $i$ and $j$ and hence the probability that there is a block of length $t$ with at least $s$ chords is bounded above by $2np(I_1)$ by the union bound. We have
\begin{align*}
2np(I_1) &\le 2n{t \choose 2s}\frac{(2s-1)!!(2n-2s-1)!!}{(2n-1)!!} \\
&\le 2n\, \frac{t^{2s}}{(2s)!}\, \frac{(2s-1)!!}{n^s}=	2n\,\frac{t^{2s}}{s!2^sn^s}		\\
&\le \lp \frac{t^2e}{2ns}\rp^s 	\le 	\lp 18n^{-2/15}\rp^{n^{1/3}} \to 0,
\end{align*}
where the third inequality follows from Stirling's formula, and the fourth one is obtained by plugging in the values of $s$ and $t$. This concludes the proof.
\end{proof}

\begin{lemma}[Step 2]\label{k_2}
As $n\to \infty$, \whp, the subdiagram $\Len{3\lf n^{1/3}\rf}$ is strongly connected.
\end{lemma}

\begin{proof}
Let $k= 3\lf n^{1/3}\rf$ and $m=\lf n^{3/5}\rf$. Let $E$ be the intersection of the events that the subdiagram $\Len{m}$ is strongly connected and there is no interval of length $3m$ that contains at least $k/3$ chords. By Lemma~\ref{k_1} and Lemma~\ref{intervals-chords}, the event $E$ holds \whp.

Conditioned on $E$, there is no chord $c$ in $D_n$ with $k \le \ell(c) \le m$ such that $c$ has fewer than $k/3$ neighbors in  $\Len{m}$. On the other hand, if $c$ is a chord with at least $k/3$ neighbors from $\Len{m}$, then the probability that $c$ over-crosses (or under-crosses) $\Len{m}$ is at most $2^{-k/3}$. Therefore, conditioned on $E$, the probability that there is a chord $c$ with $k \le \ell(c) \le m$ that does not belong to the giant component is $O\lp n\cdot 2^{-k/3}\rp$, which goes to 0. Consequently, since $E$ holds \whp, $\Len{k}$ is strongly connected \whp.
\end{proof}

\begin{lemma}[Step 3]\label{k_3}
As $n \to \infty$, the subdiagram $\Len{ 2\log n}$ is strongly connected \whp.
\end{lemma}

\begin{proof}
Let $k=3\lf n^{1/3}\rf$. By Lemma \ref{k_2}, the subdiagram $\Len{k}$ is strongly connected \whp. Thus, it is enough to show that, \whp, there is no chord $c$ with $2\log n \le \ell(c) \le k$ over-crossing (or under-crossing)~$\Len{k}$. 

First we show that there is no block in $[2n]$ of length $3k$ that contains at least $4$ chords. Indeed, the expected value of the tuples $(I,c_1,c_2,c_3,c_4)$, where $I$ is a block of length $3k$ that contains the chords $c_1,c_2,c_3$, and $c_4$  is 
\be	\label{eq: 4 chords}  
 (2n)\,\frac{{3k  \choose 8}7!!}{(2n-1)(2n-3)(2n-5)(2n-7)}= O\lp n^{-1/3}\rp.		
 \ee
Let $E'$ be the event that every chord $c$ with $2\log n \le \ell(c) \le k$ has at least $\ell(c)-4$ neighbors from  $\Len{k}$. 
By~\eqref{eq: 4 chords}, the event $E'$ holds \whp.
On the other hand, conditioned on $E'$, the probability that there is a chord $c$ with $2\log n \le \ell(c) \le k$ that 
over-crosses (or under-crosses) $\Len{k}$  is $O(n2^{-2\log n})=O(1/n)$, which finishes the proof. 
\end{proof}

\begin{lemma}[Step 4]\label{only one complex}
\Whp, there is no block of length $\lf 5\log n \rf$ that contains two chords. Consequently, \whp, there is only one strong component of size larger than 1.
\end{lemma}

\begin{proof}
Let $k=\lf 5\log n \rf$ and $X_n$ be the number of blocks of length $k$ containing two chords. Given a block $I$ of length $k$, the probability that $I$ contains two chords is at most
\[
\frac{{k \choose 4}3!!}{(2n-1)(2n-3)}= O\lp \frac{k^4}{n^2}\rp.
\]
Since there are $2n$ blocks of length $k$, the probability that there is one containing two chords is $O(k^4/n)$ by the union bound, which goes to 0. In particular, \whp, no two chords of length at most $2\log n$ cross each other.
Combining this with Lemma \ref{k_3}, \whp, there is only one strong component of size larger than~1.
\end{proof}

The last step is to find the number of strong components of size $1$. Let $F$ be the event that all the chords of length at least $2\log n$ belong to the same strong component and there is no block of size $5\log n$ containing two chords of length smaller than $2\log n$. By lemmas~\ref{k_3} and \ref{only one complex}, the event $F$ occurs \whp. 

Now, let $c$ be a chord of length $k< 2\log n$. Conditioned on the event $F$, the chord $c$ has $k$ neighbors, all from $\Len{2\log n}$. In that case, $c$ does not belong to the giant component if and only if either it over-crosses all its neighbors or it under-crosses all its neighbors. The probability that it over-crosses all $k$ of its neighbors  is $2^{-k}$, and likewise the probability that it under-crosses all $k$ of its neighbors  is $2^{-k}$.  Let $\zeta_j(i)$ be the the indicator random variable which takes the value 1 if $\ch{i,i+j+1}$ is a chord in $D_n$, and this chord either over-crosses all of its neighbors or it under-crosses all of its neighbors. Conditioned on the event $F$, the number of single-chord strong components is $\sum_{j=0}^{2\log n}\sum_{i=1}^{2n}\zeta_j(i)$. 

\begin{lemma}
For any positive integer $k\le 2\log n$, we have 
\[	\sum_{i=1}^{2n}\mean[\zeta_k(i)]= \frac{1}{2^{k-1}}+O\lp \frac{k^2}{n}\rp.	\]
\end{lemma}

\begin{proof}
Let  $p_1$ be the probability that the chord $\ch{i,i+k+1}$ exists, it has exactly $k$ neighbors, and it either over-crosses or under-crosses all of its neighbors, and let $p_2$ be the probability that $\ch{i,i+k+1}$ exists and it has fewer than $k$ neighbors. Then, we have $p_1\le \mean[\zeta_k(i)]= \pr(\zeta_k(i)=1)\le p_1+p_2$. If $\ch{i,i+k+1}$ exists and has fewer than $k$ neighbors, then there must be a chord with both endpoints in $[i+1,i+k]$. Hence the probability $p_2$ is bounded above by $ {k \choose 2}(2n-5)!!/(2n-1)!!$, so $p_2=O(k^2/n^2)$. On the other hand,
\[	 p_1= \frac{1}{2n-1}\cdot \lp 1-O(k^2/n^2)\rp\cdot \frac{1}{2^{k-1}},		\]
where the first factor accounts for the existence of the chord $c=\ch{i,i+k+1}$, the second one accounts for $c$ having exactly  $k$ neighbors, and the last one accounts for $c$ under-crossing or over-crossing its neighbors. Combining the above, we get
\[	 \mean[\zeta_k(i)]= \frac{1}{2^kn}+O(k^2/n^2).	\]
Summing over all $i$, we obtain
\[	\sum_{i=1}^{2n}\mean[\zeta_k(i)]= \frac{1}{2^{k-1}}+O\lp \frac{k^2}{n}\rp.		\qedhere		\]
\end{proof}

\begin{corollary}\label{k_4}
For any $\omega=\omega(n)$ that tends to infinity, \whp, there is no chord $c$ of length at least $\omega$ such that $c$ is a single-chord strong component.
\end{corollary}

\begin{proof}
By the previous lemma, we have
 \[  		 \sum_{j=\omega}^{2\log n}\sum_{i=1}^{2n}\mean[\zeta_j(i)]= \sum_{j=\omega}^{2\log n}\lb \frac{1}{2^{j-1}} + O(j^2/n)\rb= O\big(2^{-\omega}+\log^3n/n\big) .		\]
Hence, $ \sum_{j=\omega}^{2\log n}\sum_{i=1}^{2n}\zeta_j(i)=0$ \whp. Coupling this result with the fact that $F$ occurs with high probability, we get the desired result.
\end{proof}

\begin{theorem}
The number of single-chord strong components of $D_n$ converges in distribution to a Poisson random variable with mean $3$. \Whp, other than the single-chord strong components, there is only one strong component, which is the giant component.
\end{theorem}

\begin{proof}
The proof is an amalgamation of the previous results.  We will call a strong component of size 1 a \textit{trivial component}, a chord of length smaller than $2\log n$ a \textit{small} chord, and a chord of length at least $2\log n$ a \textit{big} chord. Let $\omega$ be an integer valued function of $n$ that tends to infinity as $n$ tends to infinity. By Lemma \ref{k_3}, Lemma \ref{only one complex}, and Corollary~\ref{k_4}, \whp, there is a giant strong component that contains all the chords of length at least $\omega$ and all other components are trivial components.
Lemma~\ref{only one complex} says that \whp\ there is no small chord $c$ crossing fewer than $\ell(c)$ big chords, where $\ell(c)$ denotes the length of $c$. Let $H$ be the intersection of all these high probability events. Conditioned on $H$, the trivial components correspond to small length chords that have as many neighbors as their lengths and they either over-cross all their neighbors or they under-cross all their neighbors. We define binomially distributed random variables
\[ \eta_i \sim \Bin(L_i,p_i),	\]
where $L_i$ denotes the number of length $i$ chords in the diagram, $p_0=1$, and $p_i=2^{1-i}$ for $i\ge 1$. Conditioned on $H$, the number of trivial components is given by $\eta_0+\cdots+\eta_{\omega-1}$. Now fix a positive integer $k$ and define
\[	\phi_{n,k}(t):=\mean\lb \exp\lp it \sum_{j=0}^k\eta_j\rp\rb.	\]
Note that $\eta_0,\dots,\eta_k$ are independent conditioned on $L_0,\dots,L_k$. Thus,
\begin{align*}
\mean\lb \exp\lp it \sum_{j=0}^k\eta_j\rp\bigg | 	L_0,\dots,L_k		 \rb 	&= 	\prod_{j=0}^{k}\mean[e^{it\eta_j}|L_j]\\
&= \prod_{j=0}^{k}\mean[e^{it\Bin(L_j,p_j)}|L_j]\\
&=\prod_{j=0}^{k}(e^{it}p_j+1-p_j)^{L_j}.
\end{align*}
Using Theorem \ref{MDP-SC},
\begin{align*}
\phi_{n,k}(t)= \mean\lb \prod_{j=0}^{k}(e^{it}p_j+1-p_j)^{L_j}\rb \to  \mean\lb \prod_{j=0}^{k}(e^{it}p_j+1-p_j)^{T_j}\rb,
\end{align*}
where $T_j$'s are independent copies of a Poisson random variable with mean $1$. On the other hand,
\begin{align*}
 \mean\lb \prod_{j=0}^{k}(e^{it}p_j+1-p_j)^{T_j}\rb &=   \prod_{j=0}^{k}\mean[(e^{it}p_j+1-p_j)^{T_j}] =\prod_{j=0}^{k}\exp(e^{it}p_j-p_j)\\
&=\exp \lb (e^{it}-1)\sum_{j=0}^{k}p_j\rb.
\end{align*}
The calculations above show that $\phi_{n,k}(t) \to \exp \lb 3(e^{it}-1)\rb$ as $k \to \infty$, where $ \exp \lb 3(e^{it}-1)\rb$ is the characteristic function of a Poisson random variable with mean 3. Thus,
\[	
\eta_0+\cdots+\eta_{k} \stackrel{d}{\longrightarrow} \Poisson(3)
\]
as $k\to \infty$.
We conclude the proof by taking $k$ to infinity first and then letting $\omega=k$.
\end{proof}


\section{Growth of a random chord diagram}\label{sec: growth of C_n}

In the previous sections we considered the static random chord diagram $C_n$. In this section we discuss a couple of dynamically growing random chord diagram models. In both models, we start with no chord and add chords one at a time randomly. In $n$ steps, we obtain a chord diagram with $n$ chords. We show that both models produce a uniformly random chord diagram after $n$ steps. Once this is established, $C_n$ can be viewed as the $n$th step snapshots of two different random processes. 

With the associated intersection graphs, the random chord diagram processes can also be viewed as graph processes. Random graph processes are studied widely in the fields of random graphs to understand the dynamics of random graphs. The best known example is the graph process $\widetilde{G}_n$, where one starts with $n$ labeled vertices and no edges, and at each step adds an edge choosing the edge uniformly at random from all non-edges. The Erd\H{o}s-R\'{e}nyi graph $G(n,m)$ can be viewed as a snapshot at time $m$ of the process~$\widetilde{G}_n$. 

After we define both models, we show that the two models are actually not much different. We use these growth models to give another extension of the result by Flajolet and Noy (Theorem~\ref{FN}).  

\subsection{Continuous model}

Consider the following growth process. Initially we have a circle with no chords and no endpoints on it. At step 1, we add a chord where we do not distinguish any two chords at this step. Thus, after step 1, we have a chord diagram with one chord. At step 2, there are three ways to draw the second chord; either the second chord does not intersect the first one, which can happen in two ways, or it intersects the first one. More specifically, the first chord separates the circle into two arcs, and either both endpoints of the second chord lie on the first arc, or they both lie on the second arc, or they lie on different arcs. One of these possibilities is chosen uniformly randomly. In general, after step $k$, there are $k$ chords and $2k$ disjoint arcs. For the next chord, there are ${2k\choose 2}+2k={2k+1\choose 2}$ ways of choosing the arcs on which the endpoints lie, and one of them is chosen uniformly at random. For convenience, which will be clear later, one of the endpoints of the first chord is labeled with $1$ after it is created. Basically, it will serve as a reference point. 

\medskip
\noindent {\bf Remarks.}
\begin{enumerate}[(1)] \itemsep=0pt \parsep=0pt
\vspace {-\topsep}
\item In this process, at each step, what matters is only the selection of the arcs for the endpoints of the chord to be added. We do not care about where exactly on the arc an endpoint is situated. 
\item We do not label the endpoints during the process except one of the endpoints of the first chord. However, after each step we know the relative ordering of the endpoints.
\end{enumerate}

Let $U_k$ denote the chord diagram obtained after the $k$th step of the above process. At this time, we can label the endpoints of the chords in $U_k$ starting with the already labeled endpoint $1$ and labeling the $j$th endpoint we encounter with $j$ as we traverse the circle clockwise. The following lemma asserts that $C_n$ is the snapshot of the above process after step $n$.

\begin{lemma}\label{continuous uniform}
$U_k$ is uniform among all the chord diagrams with $k$ chords on the set of endpoints $[2k]$.
\end{lemma}

\begin{proof}
Let $\C$ be a chord diagram with $k$ chords on $[2k]$. We need to show that $\pr(U_k=\C)= 1/(2k-1)!!$. Let $c_1,\dots,c_k$ be a labeling of the chords of $\C$ such that the chord containing endpoint $1$ is labeled as $c_1$. In the evolution of $U_k$, the chords $c_2,\dots,c_k$ can appear in any order. Let $x_2\dots x_k$ be a permutation of the chords $c_2,\dots,c_k$, which represents a particular order of appearance in the evolution. Conditioned on $\{x_1,\dots,x_{i-1}\}$, the probability that $x_i$ occurs at step $i$ is  $1/{2i-1 \choose 2}$. Thus, the probability that $x_i$ occurs at step $i$ for all $i\in [2,k]$ is
\[	\prod_{i=2}^{k}{2i-1 \choose 2}^{-1}= \frac{2^{k-1}}{(2k-1)!}.						\]
Since there are $(k-1)!$ orderings of $c_2,\dots,c_k$, we have
\[	 \pr(U_k=\C)=(k-1)!\frac{2^{k-1}}{(2k-1)!}=\frac{1}{(2k-1)!!} . 		\qedhere		\]
\end{proof}

Next we will prove an extension of Theorem~\ref{FN} which says that $C_n$ is monolithic \whp\ as $n\to \infty$.
We call $m$  a \textit{switching point} if $U_m$ is monolithic and $U_{m+1}$ is not. The next lemma shows that no switching point occurs after some point \whp.

\begin{lemma}\label{continuous-no switching}
For any $\omega \to \infty$, \whp, there is no switching point greater than $\omega$ .
\end{lemma}

\begin{proof}
Let $\xi_m$ be the indicator random variable that takes the value $1$ if $m$ is a switching point, that is, 
\[
\xi_m=
\begin{cases}
1, & \text{if } m \text{ is a switching point} \\
0 , & \text{otherwise.}
\end{cases}  
\]
We want to find an upper bound on the expected value of $\xi_m$. Let $A(m,k)$ denote the set of monolithic chord diagrams with a total of $m$ chords with $k$ of them being simple. Let $A(m)$ denote the union of $A(m,k)$ for $k\ge 0$. We have
\begin{align}\label{expected switching}
\mean[\xi_m]&= \pr(\xi_m=1)= \sum_{\C \in A(m)} \pr(U_m=\C)\cdot \pr(U_{m+1} \not \in A(m+1) | U_m= \C)		\notag	\\
&= \frac{1}{(2m-1)!!}\sum_{k=0}^{\lf 2m/3\rf} \sum_{\C \in A(m,k)} \pr(U_{m+1} \not \in A(m+1) | U_m= \C).
\end{align}

Note that the upper bound for $k$ in the outer sum above is $\lf 2m/3\rf$ since the number of simple chords in $\C$ cannot exceed $2m/3$ provided that $\C \in A(m)$. To see this, let $\C \in A(m,k)$.  The non-trivial component of $\C$ has $m-k$ chords and the endpoints of these chords determine $2m-2k$ intervals. Since each such interval can contain at most one simple chord, we have $k\le 2m-2k$ and consequently $k \le 2m/3$. 

Now suppose $U_m=\C$, where $\C \in A(m,k)$ and $\C$. In this case, $U_{m+1} \not \in A(m+1)$ if and only if the $(m+1)$st chord lies in one of the $(2m-2k)$ intervals determined by the non-trivial component and the same interval contains a simple chord. 
For each simple chord, there are $6$ possible ways to add a new chord that causes $U_{m+1}$ not belong to $A(m+1)$. (This is the number of ways to choose two endpoints from three arcs determined by the two endpoints of the existing simple chord.) Thus, 
\[ 	 \pr(U_{m+1} \not \in A(m+1) | U_m= \C))= \frac{6k}{{2m+1 \choose 2}} \le \frac{3k}{m^2}.		\]

Next we want to bound the number of summands of the inner sum in~\eqref{expected switching}, namely the size $A(m,k)$. Each $\C \in A(m,k)$ contains a connected component of size $m-k$, which  contains the endpoint $1$. Let $f(n)$ denote the number of connected chord diagrams with $n$ chords. Now the number of terms in the inner sum in~\eqref{expected switching} is $f(m-k){2m-2k \choose k}$, since we can locate the simple chords in ${2m-2k \choose k}$ ways. Hence, using the trivial upper bound $ (2m-2k-1)!!$ for $f(m-k)$,
\begin{align}\label{expected sum for xi_m}
\mean[\xi_m] 		\le
\frac{1}{(2m-1)!!}\sum_{k=1}^{\lf 2m/3\rf} (2m-2k-1)!!{2m-2k \choose k} \frac{3k}{m^2}\,.
\end{align}
Lastly, we need to bound the sum in~\eqref{expected sum for xi_m}. Let $e_k$ be the term corresponding to index $k$. Taking the quotient of two consecutive terms gives
\be \label{quo}
\frac{e_k}{e_{k+1}}= k \cdot \frac{(2m-2k)(2m-2k-1)^2}{(2m-3k)(2m-3k-1)(2m-3k-2)} \ge k.
\ee
Using~\eqref{quo} we find $(k-1)!e_k\le e_1$. On the other hand,
\be	\label{eq: 3.6.4} e_1= (2m-3)!!{2m-2 \choose 1}\frac{3}{m^2}\le (2m-1)!!\frac{3}{m^2}	.	\ee
Using equations \eqref{expected sum for xi_m}--\eqref{eq: 3.6.4}, we get 
$\mean[\xi_m]\le 3e/m^2$.
Summing over all $m$, we obtain
\[	\sum_{m = \omega}^{\infty} \mean[\xi_m] \le \sum_{m = \omega}^{\infty} \frac{3e}{m^2}= O(\omega^{-1}),	\]
which finishes the proof.
\end{proof}

\begin{corollary}\label{cor: continuous model-monolithicity}
Let $m\to \infty$. \Whp, all the diagrams in $\{U_n\, : n\ge m\}$ are jointly monolithic.
\end{corollary}

\begin{proof}
This follows immediately from Theorem \ref{FN} and Lemma \ref{continuous-no switching}.
\end{proof}

\subsection{Discrete model}

Now we present another model for an evolution of a random chord diagram. As in the previous section, we add the chords one at a time. However, in this case, we start with $2n$ given points on a circle, labeled $1$ through $2n$ clockwise in increasing order. These points constitute a universal set for the endpoint set of a chord diagram obtained during the process. At the first step we create the first chord by choosing a partner uniformly at random for the endpoint 1. Then,  at each later step we add a chord by choosing two of the yet unused endpoints and pairing them up. Thus, at step $k$, we add a chord by joining two of the $2n-2k+2$ endpoints, which do not belong to any of the first $k-1$ chords. We denote by $C_k'$ the random chord diagram obtained after the $k$th chord is added in this process. Now we define a relabeling operation for chord diagrams.

\begin{definition}
(\textit{$\tau$ operator})
Let $\C$ be a chord diagram with $k$ chords whose endpoint set $\{i_1,\dots,i_{2k}\}$ is a subset of $[2n]$, and let $i_1<i_2\cdots<i_{2k}$. We denote by $\tau(\C)$ the chord diagram obtained from $\C$ by relabeling $i_t$ as $t$ for every $t\in [2k]$.   
\end{definition}

\begin{example*}
Let $n=5$ and let $\C$ be the chord diagram consisting of the chords $(2,8), (3,5)$, and $(4,9)$. Then $\tau(\C)$ is the chord diagram consisting of the chords $(1,5), (2,4)$, and $(3,6)$.
\end{example*}

\begin{lemma}
$\tau(C_k')$ is uniformly distributed among all the chord diagrams with $k$ chords on the endpoint set $[2k]$.
\end{lemma}

\begin{proof}
Let $\C$ be a chord diagram on $[2k]$ with $k$ chords. Let $c_1,\dots,c_k$ be the chords of $\C$ where $c_1$ contains the endpoint 1. For any endpoint set $I=\{1=i_1,i_2,\dots,i_{2k}\} \subset [2n]$, there is a unique chord diagram $\C'$ on  $I$ such that $\tau(\C')=\C$. On the other hand, for any $\C'$ with $\tau(\C')=\C$, 
\[	\pr(C_k'=\C') = \frac{1}{2n-1}\cdot \prod_{j=1}^{k-1}\frac{k-j}{{2n-2j \choose 2}}		\]
since the partner of endpoint 1 in $\C'$ is chosen in the first step with probability $1/(2n-1)$, and 
at step $(j+1)$ we can choose one of $(k-j)$ chords out of ${2n-2j \choose 2}$ chords for $j\in [k-1]$. Since there are ${2n-1 \choose 2k-1}$ diagrams $\C'$ such that $\tau(\C')=\C$ (one for any choice of the $2k-1$ endpoints), we have
\[	 
\pr(\tau(C_k')=\C) = {2n-1 \choose 2k-1}\cdot \frac{1}{2n-1}\cdot \prod_{j=1}^{k-1}\frac{k-j}{{2n-2j \choose 2}}= \frac{1}{(2k-1)!!},
 \]
 which gives the desired result.
\end{proof}

By applying the $\tau$ operator to the chord diagrams obtained during the process, we get a sequence of chord diagrams, $\tau(C_1'),\dots,\tau(C_n')$, where $\tau(C_k')$ is uniformly distributed over all chord diagrams with $k$ chords on $[2k]$. Now, we want to show that this sequence has the same probability distribution with the sequence of the first $n$ diagrams obtained in the previous section. We write $X \stackrel{d}{=}Y$ when the two random variables $X$ and $Y$ have the same distribution. 

\begin{lemma}\label{lemma: same distribution}
$(\tau(C_1'),\dots,\tau(C_n')) \stackrel{d}{=} (U_1,\dots,U_n)$, that is, 	
\[	 (C_1',\dots,C_n') \stackrel{d}{=} (U_1,\dots,U_n)	\]
\end{lemma}

Before proving the lemma, we introduce another operation in chord diagrams.

\begin{definition}
\textit{($\varphi$ operator)}
Let $\C$ be a chord diagram with $n$ chords on the endpoint set $[2n]$. Let 
\[ 	r: \{\text{chords of } \C \}	 \rightarrow [n]	\]
be a bijection such that $r(c_1)=1$, where $c_1$ denotes the chord containing the endpoint 1. Hence, $r$ defines a labeling of the chords such that the chord $r^{-1}(k)$ is labeled with $k$. Let $\C_j$ be the subdiagram consisting of the chords $\{r^{-1}(1),\dots,r^{-1}(j)\}$, the set of chords of $\C$ labeled with $\{1,\dots,j\}$ under $r$. Then we define
\be \label{twosequences}	\varphi_{n,j}^{r}(\C)= \tau(\C_j).					\ee
We supress the subscript $n$ and the coloring $r$ when there is no danger of confusion and write $\varphi_j(\C)$ instead of $\varphi_{n,j}^{r}(\C)$. In words, $\varphi_j$ deletes the last $n-j$ chords and relabels the endpoints of the remaining chords with the labels in $[2j]$, respecting their previous relative ordering.
\end{definition}

\begin{example*}
Let $n=7$. Let $\C$ be the chord diagram on the endpoint set $[14]$ with the chords $\ch{1,5}$, $\ch{2,14}$, $\ch{3,8}$, $\ch{4,9}$, $\ch{7,12}$, $\ch{6,13}$, and $\ch{10,11}$. Let $r$ be the labeling of the chords with the labels 1 through 7, in this given order.
Then, $\varphi_{7,4}^r(\C)$ is a chord diagram of size 4, and its chords are $\ch{1,5}$, $\ch{2,8}$, $\ch{3,6}$, and~$\ch{4,7}$.
\end{example*}

\begin{proof}[Proof of Lemma \ref{lemma: same distribution}]
Let $D_1,\dots,D_n$ be given chord diagrams such that $D_i$ is a chord diagram with $i$ chords on $[2i]$. We want to show that 
\be \label{UC'}	\pr(\tau(C_1')=D_1,\dots,\tau(C_n')=D_n)= \pr(U_1=D_1,\dots,U_n=D_n)			\ee
For brevity, we write $\mathbf{D}$ for $(D_1,\dots,D_n)$, $\mathbf{\tau(C')}$ for $(\tau(C_1'),\dots,\tau(C_n'))$, and $\mathbf{U}$ for $(U_1,\dots,U_n)$. In both processes, $\{U_k\}_{k=1}^n$ and $\{C_k'\}_{k=1}^n$, we label the chord created at the $j$th step with $j$. We denote by $r_c$ the chord labeling of $C_n'$ obtained during the process $\{C_k'\}_{k=1}^n$, and by $r_u$  the chord labeling of $U_n$ obtained during the process $\{U_k\}_{k=1}^n$. By definition, we have $\varphi_{n,j}^{r_c}(C_n')=\tau(C_j')$ and $\varphi_{n,j}^{r_u}(U_n)=U_j$ for all $j\in [n]$.
 
Let $L=L(\mathbf{D})$ be the set of labelings of the chords of $D_n$ by $[n]$, such that the chord containing endpoint 1 is labeled with $1$ and $\varphi_j(D_n)=D_j$ for all $j\in [n-1]$. Note that, if $L=\emptyset$, then there cannot be any evolution in either model that produces $\mathbf{D}$. Consequently, both sides of Equation \eqref{UC'} are $0$ in that case. 

\medskip
\textit{Claim}: The sequence $\mathbf{\tau(C')}$  is the same as $\mathbf{D}$ if and only if $C_n'=D_n$ and $r_c \in L$. Similarly, the sequence $\mathbf{U}$ is the same as $\mathbf{D}$ if and only if $U_n=D_n$ and $r_u \in L$.

\textit{Proof of the claim}:
We prove only the first part of the claim since the second part is very similar. Suppose first that  $\mathbf{\tau(C')}=\mathbf{D}$. In this case, the last components of the two vectors must be the same and hence $C_n'=D_n$. Consequently,
\[	
\varphi_{n,j}^{r_c}(D_n)=\varphi_{n,j}^{r_c}(C_n')=\tau(C_j')= D_j \text{ for all } j\in [n],
\]
from which it  follows that $r_c \in L$. Conversely, if $D_n=C_n'$ and $r_c \in L$, then obviously
\[	D_j = 	\varphi_{n,j}^{r_c}(D_n)=\varphi_{n,j}^{r_c}(C_n')=\tau(C_j'),	\]
which finishes the proof. \qed

Note that a necessary condition for two diagrams having the same chord labeling is that the diagrams are the same. Thus, the above claim could indeed be stated as 
\begin{enumerate}[(i)] \itemsep0pt \parsep0pt \parskip0pt
\vspace{-\partopsep}
\item $\mathbf{\tau(C')}= \mathbf{D} \iff 	r_c \in L$, and
\item 	$ \mathbf{U}= \mathbf{D}  \iff 	r_u \in L$.
\end{enumerate}
Now, in order to find the probabilities on either side of Equation \eqref{UC'}, we need to find the probability that a given $r\in L$ is produced during the corresponding process, i.e., the probabilities $\pr(r_c=r)$ and $\pr(r_u=r)$.
Let $r$ be a labeling in $L$. The two labelings $r_c$ and $r$ are the same if and only if  the chord $r^{-1}(i)$ is created at the  $i$th step of the process $\{C_k'\}_{k=1}^n$ for all $i \in [n]$. Therefore, we need to find the probability that the chord created at step $i$ is the same as $r^{-1}(i)$  conditioned on the previous $i-1$ chords. Now, the chord labeled $1$ is created at the first step with probability $1/(2n-1)$, and conditioned on the first $i-1$ chords, the chord labeled $i$ is created at the $i$th step with probability $2/(2n-2i+2)(2n-2i+1)$ for $i>1$. Thus,
\[	\pr(r_c=r)=  \frac{1}{(2n-1)}\prod_{i=2}^{n}\frac{2}{(2n-2i+2)(2n-2i+1)}= \frac{ 2^{n-1}}{(2n-1)!},	\]
and consequently
\[	\pr(\mathbf{\tau(C')}=\mathbf{D})= |L| \cdot \frac{1}{(2n-1)}\prod_{i=2}^{n}\frac{2}{(2n-2i+2)(2n-2i+1)}= \frac{|L|\cdot 2^{n-1}}{(2n-1)!}.		\]

We argue similarly for the right side of Equation~\eqref{UC'}. More specifically, for the equality $U_j=D_j$ for all $j\in [n]$ to hold, we must have $r_u=r$ for some $r\in L$. 
In $D_n$, or in $U_n$,  we do not have to keep a record of the endpoint labels other than the endpoint labeled $1$. 
As long as we know the endpoint labeled 1 and the positions of the endpoints of the chords relative to each other, we can uniquely determine the endpoint labeling for $D_n$. 
Now, given a labeling $r$  in $L$, we need to find the probability that $r$ is obtained during the process $\{U_k\}_{k=1}^n$. For the first chord there is no restriction, so it will be created with probability 1. 
The chord labeled 2 in $r$ is created with probability $1/3$ at the second step of the process $\{U_k\}$ since there are three possibilities after the first chord is drawn. 
In general, if $k$ chords are already created in $\{U_k\}_{k=1}^n$ matching the labeling $r$, the chord labeled as $k+1$ will be created with probability $\frac{1}{{2k+1 \choose 2}}$ at the $(k+1)$st step. Thus, $r$ is obtained with probability
\[		\prod_{k=1}^{n-1} \frac{1}{{2k+1 \choose 2}} =\frac{2^{n-1}}{(2n-1)!},			\]
and consequently,
\[ \pr(\mathbf{U} =\mathbf{D}) = \frac{|L|\cdot 2^{n-1}}{(2n-1)!}.	\]
Thus, $\pr(\mathbf{\tau(C')}=\mathbf{D})=\pr(\mathbf{U} =\mathbf{D})$ as desired.
\end{proof}

We extend the definitions of simple chords and monolithicity to chord diagrams obtained during the second process $\{C_k'\}_{k=1}^n$. We say that $C_k'$ is monolithic if $\tau(C_k')$ is monolithic. Combining Corollary \ref{cor: continuous model-monolithicity} with Lemma \ref{lemma: same distribution}, we obtain the following generalization of the Flajolet-Noy result (see Theorem \ref{FN}).

\begin{corollary}\label{continuous model-monolithicity}
Let $\omega(n)$ be a function of $n$ that tends to infinity as $n$ tends to infinity. \Whp, all the chord diagrams $C_m'$ for $\omega(n) \le m\le n$ are monolithic as $n \to \infty$. \qed
\end{corollary}


\section{Concluding remarks}\label{sec:conclusion}

Let $\C$ be a chord diagram and $G_{\C}$ be the corresponding intersection graph. We define the cliques and independent sets of $\C$ as the preimages of the cliques and independent sets in $G_{\C}$, respectively. Hence, a {\em clique} in $G_{\C}$ corresponds to a set of pairwise intersecting chords in $\C$. Likewise, an {\em independent set} in $G_{\C}$ corresponds to a set of pairwise non-intersecting chords. The {\em clique number} of $\C$ is the largest $k$ for which there is a $k$-clique in $\C$. Similarly, the {\em independence number} of $\C$ is the largest $k$ such that there is an independent set of size $k$ in $\C$.

Recall from Section~\ref{sec:intro} that Chen et al.~\cite{Chen07} and Baik and Jenkins~\cite{Ba-Jen13} defined an $r$-crossing and an $r$-nesting as follows. A set of  $r$ chords $\ch{x_1,y_1},\dots,\ch{x_r,y_r}$ is an $r$-crossing if $x_1<\cdots<x_r<y_1<\cdots<y_r$ and it is an $r$-nesting if $x_1<\cdots<x_r<y_r<\cdots<y_1$. With this definition, an $r$-crossing is the same as an $r$-clique. On the other hand, an $r$-nesting is an independent set of size $r$ but the converse is not true, that is, not every independent set in $\C$ is a nesting. 

Now let $\omega(C_n)$ and $\alpha(C_n)$ denote the clique number and the independence number of $C_n$, respectively. Also, let $\alpha'(C_n)$ denote the nesting number of $C_n$, the largest number $r$ for which there is an $r$-nesting. Using a result of  Baik and Rains~\cite[Theorem 3.1]{Ba-Ra01}, Chen et al.~\cite[Remark 5.6]{Chen07} remarked that
\be \label{clique dist.}
\lim_{n \to \infty} \pr \lp \frac{\omega(C_n)-\sqrt{2n}}{(2n)^{1/6} } \le 2x\rp = F(x),
\ee
where $F(x)$ is the GOE Tracy-Widom distribution function. Moreover, they proved that $\alpha'(C_n)$ and $\omega(C_n)$ are equally distributed, so~\eqref{clique dist.} holds if we replace $\omega(C_n)$ by $\alpha'(C_n)$ in the equation. On the other hand, we have $\alpha(C_n)\ge \alpha'(C_n)$ since each nesting is an independent set. 

Now define $X_r$ as the number of independent sets of size $r$ in $C_n$. We have
\[
\mean[X_r] = \frac{{2n \choose 2r}(2n-2r-1)!!}{(2n-1)!!} \cdot \frac{1}{r+1}{2r \choose r}
\]
since for each set of $2r$ points there are ${2r \choose r}/(r+1)$ ways to pair them without any intersection (this number is the $r$th Catalan number). For $r=e\sqrt{2n}$, simplifying this expression and using Stirling's formula, we get
\[
\mean[X_r] \sim \frac{1}{2\pi r^2} \lp \frac{2e^2n}{r^2}\rp^r e^{-r^2/2n} 	\to 0.
\]
This equation, together with~\eqref{clique dist.} and the fact that $\alpha(C_n)\ge \alpha'(C_n)$, show that, \whp,
\[
\sqrt{2n}-tn^{1/6} \le \alpha(C_n) \le e\sqrt {2n}
\]
for any $t\to \infty$. We conclude this work with two open questions.

\begin{question}
Is there a constant $\beta$ such that $\alpha(C_n)/ \sqrt{n} \to \beta$ in probability?
\end{question}

\begin{question}
 What is the asymptotic distribution of $\alpha(C_n)$?
\end{question}


\end{document}